\documentclass[11pt]{amsart}
\makeatletter
\def\ps@plain{%
  \def\@oddfoot{\hfil\thepage\hfil}%
  \def\@evenfoot{\hfil\thepage\hfil}%
  \def\@oddhead{}%
  \def\@evenhead{}%
}
\usepackage {fullpage}
\makeatother
\usepackage[utf8]{inputenc}
\usepackage{amssymb}
\usepackage{amsmath}
\usepackage{float}
\usepackage{graphicx}
\usepackage{subcaption}
\usepackage{tikz-cd}
\usepackage{amsthm}
\usepackage[colorlinks=true, allcolors=blue]{hyperref}
\usepackage[letterpaper,top=2cm,bottom=2cm,left=3cm,right=3cm,marginparwidth=1.75cm]{geometry}
\usepackage{enumitem}   
\usepackage{cite}
\newtheorem{theorem}{Theorem}

\newtheorem{lemma}[theorem]{Lemma}

\newtheorem{proposition}[theorem]{Proposition}

\newtheorem{remark}[theorem]{Remark}

\newcommand{\R}{\mathbb{R}}

\newcommand{\eps}{\varepsilon}

\title{Singularity of the loops within a cable-graph loop-soup conditioned by its occupation time}
\author{Arthur Dremaux}
\address{University of Cambridge}
\email {ad2231@cam.ac.uk}
\date{}

\begin{document}

\begin{abstract}
In this note, we show the following feature of the relation between Brownian loop-soups on cable-graphs and their total occupation time-field  $\Lambda$:  When conditioned on $\Lambda$, the conditional law of individual loops becomes singular with respect to that of unconditioned loops. The idea of the proof is to see that some type of fast points on the curve $\Lambda$ impose an exceptional behaviour of all the loops when they go through these points.
\end{abstract}

\maketitle

\section{Introduction}

Path properties of real-valued Brownian motion have been a constant source of study in this past century. Recall that almost surely, Brownian motion is a random continuous function with quadratic variation equal to $t$ on $[0,t]$, that is differentiable nowhere and does not have points of local increase. Since Khinchin-Kolmogorov's law of the iterated logarithm and the various results by Paul L\'evy \cite{levy1992processus} (that include his result about the modulus of continuity of Brownian motion which will play a role in the present paper), many fascinating features of Brownian paths have been described -- see for instance \cite {mp2010} and the references therein.

Another of these properties that will be relevant in the present paper is the existence of continuous local time profile for one-dimensional Brownian paths: Almost surely, for each finite time $T$, the occupation time measure $\mu_T$ of a Brownian motion $\beta=(\beta_t, t \ge 0)$ up to time $T$ that is defined by $\mu_T (A) = \int_0^T 1_{\beta_s \in A} ds$, has a continuous density $x \mapsto \ell_x^T$ with respect to the Lebesgue measure (see again \cite{mp2010} or \cite{revuzyor}).

It is a rather natural question to investigate the conditional law of $(\beta_t, t \le T)$ when one conditions on its occupation time profile at a positive time $T$. This is the so-called Brownian burglar that was introduced and studied by Warren and Yor in  \cite{SPS_1998__32__328_0} -- the case where $\beta$ is a Brownian excursion has then been studied  by Aldous in \cite{MR1650567}.  More recently, A\"\i d\'ekon, Hu and Shi \cite {AHS} did use another approach based on the joint law of all local times profiles (a little bit in the spirit of the construction of the ``true self-repelling motion'' in \cite {TW}) to construct the process (which will be rather closely related to the setups that we will be studying).
In other words, one conditions the Brownian path $\beta$ by where/how it ``spreads'' its time until $T$. We can note that any almost sure properties of the Brownian paths (like abouts its quadratic variation, the modulus of continuity etc) will necessarily still hold for this conditioned process (indeed, an event that holds almost surely will almost surely hold almost surely for the conditional law), so that the trajectory of a  Brownian burglar looks a priori similar to that of an usual Brownian path.
However, as we will see in the present paper,  if $T$ is the first time at which the local time at the origin of $\beta$ reaches some positive value $l$, then it is possible to argue that for all $\eps >0$, the conditional law of $(\beta_t, t \le \eps)$ given $\mu_T$ is singular with respect to that of an unconditioned Brownian motion up to time $\eps$ (this will be Proposition \ref {Pburglar}).

One natural avenue to derive this type of result is to look at the Brownian burglar as a Brownian motion with a ``drift'', and to check that this drift is almost surely not in the Cameron-Martin space of Brownian motion.
While this strategy can actually be made to work and provides informative insight into the properties of the Brownian burglar (this will actually be the focus of our upcoming paper \cite {D}), it requires  rather subtle considerations. In the present paper, we will instead use a simple rather direct alternative approach, that roughly amounts to showing that the law of the occupation time measure of the burglar at time $\eps$ has to be singular with respect to that of  unconditioned Brownian motion -- which then implies the proposition. This approach is in fact in the spirit of (though not identical to) a number of works from the mid-1980's about exceptional times and points for Brownian trajectories. For instance, Barlow and Perkins \cite{BarlowPerkins1984, BarlowPerkins1986} did study some exceptional points at which Brownian motion behaved exceptionally ``at each visit'' and studied how it impacted its local time profile (see the references therein for further related work).

\medbreak

Our main purpose in the present paper is in fact to study this type of question in the closely related set-up of Brownian loop-soups in transient cable-graphs that we now briefly recall. A cable-graph $G$ can be viewed as the metric space obtained when one first considers a discrete graph, then associates to each edge $e$ joining two sites $x$ and $y$ in $G$ a segment $I_e$ isomorphic to some compact interval of $\R$, where the endpoints of the edges correspond to $x$ and $y$. The cable-graph is then the union of these ``physical'' edges $I_e$. On this graph, it is then possible to define Brownian motion (loosely speaking, when it is in the middle of an edge, it moves like a one-dimensional Brownian motion, and when it is at a vertex, it chooses for each excursion away from $x$ uniformly at random in which edge adjacent to $x$ to start). The trace on the sites of the graph of such a Brownian motion can be viewed as a random walk of $G$. One can then also define a natural measure on Brownian loops on such a cable-graph, and then consider a Poisson point process of Brownian loops  with intensity given by some constant $c$ times this Brownian loop measure (see \cite{Lupu2016LoopClustersInterlacementFreeField} for background). When the graph $G$ is transient, this is a Poisson point process of (unrooted) Brownian loops where each given edge is almost surely visited by infinitely many loops, but only by finitely many ones with diameter greater than any fixed $\eps$. Each of these Brownian loops in the loop-soup has a continuous occupation time density, and when one sums these densities over all Brownian loops in the loop-soup, one obtains the occupation time density of the entire Brownian loop-soup. This is a random function $\Lambda_c$ that can be viewed as the cable-graph generalization of the squared Bessel processes of dimension $c$. In particular, when $c=1$, the law of $\Lambda_c$ is that of the square of a Gaussian Free Field on this cable-graph (recall that the GFF on the cable-graph is a continuous function from the cable-graph into $\R$ that can be viewed as the natural way to extend Brownian bridges on each interval $I_e$ to the whole cable-graph).

It is natural and intriguing to investigate the relation between the occupation time and its decomposition into loops ; particularly in the cases where $c=1$ and $c=2$, where the laws of the loop-soup are invariant under certain simple local rewiring operations (for $c=1$ one has to consider unoriented loops, and for $c=2$, one considers oriented loops) that do preserve its occupation time field  \cite{werner2015spatialmarkovpropertysoups}. Aspects of the conditional law of the loop-soups given $\Lambda_c$ in the case $c=1$ have for instance been discussed in the recent papers \cite {werner2025switchingidentitycablegraphloop,lupu2025intensitydoublingbrownianloopsoups}.

The main purpose of the present note is to provide a simple proof of the following fact for any positive $c$:
\begin {proposition}
 \label {Ploops}
When one conditions on the total occupation time field $\Lambda_c$, the conditional law of each individual loop in the loop-soup  is singular with respect to that of an unconditioned loop (i.e., it is singular with respect to the loop measure).
\end {proposition}
In other words, if one knows $\Lambda_c$, then by looking at  any loop, one can detect that it is not an ``unconditioned loop''. As we will see, it is in fact sufficient to look at any portion of any loop.

As we shall see, this is closely related to the previous burglar question.
The analogue of the Brownian burglar process in the case $c=1$ is the scaling limit of the inverse VRJP process studied by Lupu, Sabot and Tarr\`es in \cite{sabot2015invertingrayknightidentity,Lupu_2019,Lupu_2019b,Lupu_2021}. Here also, one strategy would be to study this process and its dynamics in detail and to show that it can be viewed as a Brownian motion with a drift that does not belong to the Cameron-Martin space, but we will instead use a softer approach based on the local time profiles only, which will work for all $c>0$ -- so that we will not build on this work by Lupu, Sabot and Tarr\`es.

\medbreak

The spirit of this note is to provide a short essentially self-contained proof of Proposition \ref {Ploops}. We will therefore derive only the necessary intermediate results, leaving their generalizations as remarks. We will first recall some features about fast points of Brownian motion (Section~\ref {fasttimes}) and discuss aspects of their generalizations to other diffusions (Section~\ref {quicktimes}). Building on this, we will then prove Proposition \ref {Ploops} and the result about the Brownian burglar in Sections~\ref{proof1} and~\ref{proof2}.

\section{Fast points of Brownian motion}
\label{fasttimes}

Throughout this section, we will use the notation
$$ U(h) := \sqrt { 2 h \log (1/h)}.$$

Let $B$ be a standard Brownian motion [we will use the letter $B$ to denote Brownian motions that will be related in some way to the local time processes in space, so we will denote its variables by $x$, and the letter $\beta$ for Brownian paths parametrized by time $t$ -- even if both $t$ and $x$ can end up being real-valued]. Recall on the one hand the (local) law of the iterated logarithm states that for each given $x$, one has almost surely
$$\limsup_{h \downarrow 0} \frac {B(x+h) - B(x)}{\sqrt {2 h \log \log (1/h) }} = 1,$$
while L\'evy's global modulus of continuity result states that almost surely, for all $T$
$$\limsup_{h \downarrow 0} \sup_{0 \le x ,y\le T, \ |x-y| \le h}
\frac{|B(x) - B(y)|}{U(h)} = 1.$$
More generally, Orey and Taylor  \cite{OreyTaylor1974} did then study the exceptional times, called \textit{fast times} or \textit{fast points} of Brownian motion where the local modulus of continuity is exceptionally large.
If $F$ is a function, then for each $a >0$, one can say that $x$ is an $a$-fast point of $F$ if
$$\limsup_{h \downarrow 0} \frac{|F(x + h) - F(x)|}{U(h)} \ge a.$$
One can then define the set of $a$-fast points of a Brownian path $B$ to be $A_a$. We then define the set $A$ of all fast points of $B$ (when we do not specify any value of $a$) to be  $A = \cup_{a>0} A_a$.
It is simple to see that for each $a$, the set $A_a$ is measurable with respect to the $\sigma$-field generated by Brownian motion as it can be defined as the countable union of countable intersections of countable unions of countable intersections of measurable sets.
L\'evy's modulus of continuity result clearly implies that $A_a$ is almost surely empty for all $a>1$. On the other hand Orey and Taylor \cite{OreyTaylor1974} showed  that almost surely, for each $a \in (0,1)$, the set $A_a$ is not empty (and therefore dense) and has Hausdorff dimension equal to $1- a^2$.

In our proof of Proposition \ref {Ploops}, we will only use the following two facts for Brownian paths:
\begin {itemize}
 \item Almost surely, the set $A_a$ is empty for all $a>1$.
 \item Almost surely, the set $A_1$ is not empty and dense.
\end {itemize}
As we have already mentioned, the first fact follows from L\'evy's modulus of continuity. One way to derive the second one is to combine the fact that $A_a$ is non-empty and dense for all $a<1$ with Baire's Category Theorem (this is explained in the textbook \cite{mp2010} that one can also consult for all these facts).
One can also note that it is possible to instead go back to the actual proof of L\'evy's modulus of continuity and construct ``by hand'' points that belong to $A_1$ simply using Borel-Cantelli.

\medbreak

As a warm-up to what will follow, we can make the following remark: Suppose that $B$ and $B'$ are two independent Brownian motions. Then, $Z:= (B+B') / \sqrt {2}$ is also a Brownian motion.
Clearly, any $1$-fast point $x$ for $Z$ that is not a fast point for $B$ has to be a $\sqrt {2}$-fast point for $B'$, and we know that almost surely no such points exist. It therefore follows that almost surely, any $1$-fast point of $Z$ is necessarily a fast-point of both $B$ and $B'$.
It is in fact not difficult to show that the set of $1$-fast points of $Z$ is exactly the intersection of the set of  $(1/\sqrt {2})$-fast points of $B$ with the set of $(1/\sqrt {2})$-fast points of $B'$ (noting in particular that these two independent sets have Hausdorff dimension $1/2$, while for any $a < 1/ \sqrt 2$, the intersection of $A_a$ with $A_{\sqrt{2}-a}'$ is almost surely empty), but this won't be necessary here.

\section{``Quick points'' of a squared Bessel process}
\label {quicktimes}
Fast points are usually studied for Brownian motion, but it is actually easy to deduce similar results for solutions of SDEs. Indeed, the local behaviour of the solution to an SDE with continuous coefficients like $dX_x = \sigma (X_x) dB_x + b(X_x) dx$ at time $x$ where $\sigma (X_x)$ is not equal to $0$ will be that of $\sigma (X_x)$ times the behaviour of the driving Brownian motion $B$ at that time.  For instance, when $X$ is a squared Bessel process of dimension $n \ge 0$ (here and in the sequel, $n$ is not necessarily an integer), i.e., the solution to
$$\mathrm{d}X_x = 2\sqrt{X_x}\mathrm{d}B_x + n\mathrm{d}x$$
that is absorbed (or possibly reflected, if $n \ge 1$) at its first hitting time $\sigma$ of $0$,
one can define its set of generalized $a$-fast points -- that we will refer to as ``$a$-quick points'' or ``$a$-quick times'' to be the intersection of $\{ x \ge 0 \ : \ X_x \not= 0 \}$ with the set of $a$-fast points of $B$. More generally, when $F$ is a continuous function, one can define its set of $a$-quick points as
    $$Q_a (F) := \Bigl\{ x \ge 0\  : \limsup_{h\downarrow 0}\frac{|F(x+h)- F(x)|}{U(h)} \ge  2a\sqrt{F(x)} > 0\Bigr\}. $$
    The set of quick points of $A$ (when we do not specify any value of $a$) is once again defined to be $Q(F):= \cup_{a>0} Q_a(F)$, i.e., the set of fast points of $B$ at which $X$ is not equal to $0$.
    By the results we have about fast points of the Brownian motion $B$, we immediately deduce that for a squared Bessel process $X$, $Q_a(X)$ is almost surely empty for all $a>1$, and $Q_1(X)$ is almost surely non-empty and dense in the set of points at which $X$ is positive.

This leads naturally to the question about quick points of sums of squared Bessel processes.
A very simple deterministic result that will be useful and sufficient for us is the following:
\begin{lemma}
\label{Ladding}
Suppose that $F_1$ and $F_2$ are two functions such that $F:=F_1 + F_2$ has $1$-quick points and no $a$-quick points for any $a>1$, and that $F_1$ has also no $a$-quick points for any $a >1$. Then, every $x$ for which $F_2(x)>0$, $F_1 (x)>0$ and that is a $1$-quick point of $F$ is necessarily a quick point of $F_2$.
\end{lemma}

\begin{proof}
Suppose that $x \in Q_1 (F)$, i.e., that
$$ \limsup_{h\downarrow 0}\frac{|F(x+h)- F(x)|}{U(h)} = 2\sqrt{F(x)} > 0$$
Since $F_1$ has almost surely no $a$-quick point with $a>1$, it follows that for all $y$ (and in particular at $x$) for which $F_1 (y) >0$,
$$\limsup_{h\downarrow 0}\frac{|F_1(y+h)- F_1(y)|}{U(h)} \leq 2\sqrt{F_1(y)} $$
But since $F=F_1+F_2$, we get that
    $$ \limsup_{h\downarrow 0}\frac{|F(x+h)- F(x)|}{U(h)} \leq \limsup_{h\downarrow 0}\frac{|F_1(x+h)- F_1(x)|}{U(h)} +\limsup_{h\downarrow 0}\frac{|F_2(x+h)- F_2(x)|}{U(h)},
    $$
and therefore
    $$\limsup_{h\downarrow 0}\frac{|F_2(x+h)- F_2(x)|}{U(h)} \geq 2\left(\sqrt{F(x)} - \sqrt{F_1(x)}\right). $$
This last difference is clearly positive when $F_2 (x)$ is positive because  $F(x)-F_1(x)= F_2 (x) >0$, which concludes the proof.
\end{proof}

An example is for instance to consider $F_1$ and $F_2$ to be two independent squared Bessel processes of respective dimensions $0$ and $n \ge 0$. The process $F = F_1 + F_2$ is then also a squared Bessel process of dimension $n$, so that the conditions of the lemma are almost surely fulfilled.

\section {Proof of Proposition \ref {Ploops}}
\label {proof1}

Let us now return to the setup of a loop-soup in a transient cable-graph. We consider a Poisson point process ${\mathcal L}$ of Brownian loops with positive intensity $c$ times the ``natural'' Brownian loop measure (we use here the convention from \cite {MR2045953,werner2021lecturenotesgaussianfree}, so this intensity is twice the intensity that is called $\alpha$ in \cite {LeJan2011,Lupu2016LoopClustersInterlacementFreeField}). Throughout this section, $c$ will be fixed, and we will omit the subscripts that indicate the dependence on $c$ (so we will write $\Lambda$ instead of $\Lambda_c$). We consider a fixed large but compact subset $K$ of the cable-graph (if the cable-graph is compact, we would naturally take $K$ to be the entire cable-graph). Almost surely, for any $\eps >0$, the cardinality $N_\eps (K)$ of the set ${\mathcal L}_{\eps, K}$ of Brownian loops of diameter greater than $\eps$ that do intersect $K$ is finite, so that we can order all the loops that intersect $K$ in decreasing order of diameter as $\beta_1, \beta_2, \ldots$ We denote by $\Lambda$ the total occupation time density of the loop-soup (i.e., $\Lambda(x)$ is the sum of the local time at $x$ over all loops in the loop-soup). So, when $c=1$, $\Lambda$ is distributed as the square of a GFF on the cable-graph. In the general case, the law of $\Lambda$ is the natural generalization to the cable-graph of the square of a Bessel process of dimension $c$. In particular, if one conditions on the value of $\Lambda$ at two neighboring vertices on the graph, the conditional law on the edge between these two vertices will be that of a squared Bessel bridge of dimension $c$. This provides one simple way to see that almost surely, the set of $1$-quick points of $\Lambda$ is dense on the cable-graph and that $\Lambda$ has no $a$-quick points for $a>1$,

When sampling the loop-soup ${\mathcal L}$, one can first sample ${\mathcal L} \setminus {\mathcal L}_{\eps, K}$ and then the (independent) collection ${\mathcal L}_{\eps, K}$.
We can observe that for each $K$, the probability that $N_\eps(K) = 0$ is strictly positive. So, the probability that the total occupation time $\tilde \Lambda$ of the former is equal to $\Lambda$ is positive (independently of $\tilde \Lambda$). This immediately shows that almost surely, for every rational $\eps$, the occupation time $\tilde \Lambda$ has a dense set of $1$-quick points, and no $a$-quick points for $a>1$. It then easily follows that the same holds for the occupation time measures of ${\mathcal L} \setminus \{ \beta_1, \ldots, \beta_n \} $ for all $n \ge 1$ (because $\{ \beta_1, \ldots, \beta_n \}$ will be equal to ${\mathcal L}_{\eps, K}$ for some rational $\eps$).
A little variation of this argument shows that the same holds for the occupation time measure $\Lambda_n$ of ${\mathcal L} \setminus \{ \beta_n \}$. [One just samples first the loop-soup without the loops of diameter in $[a_1, a_2]$ that intersect $K$ and then the collection of such loops -- and almost surely, $\beta_n$ will be the only missing loop for some rational $a_1 < a_2$].

\begin{figure}[H]

    \centering
    \includegraphics[width=0.9\textwidth]{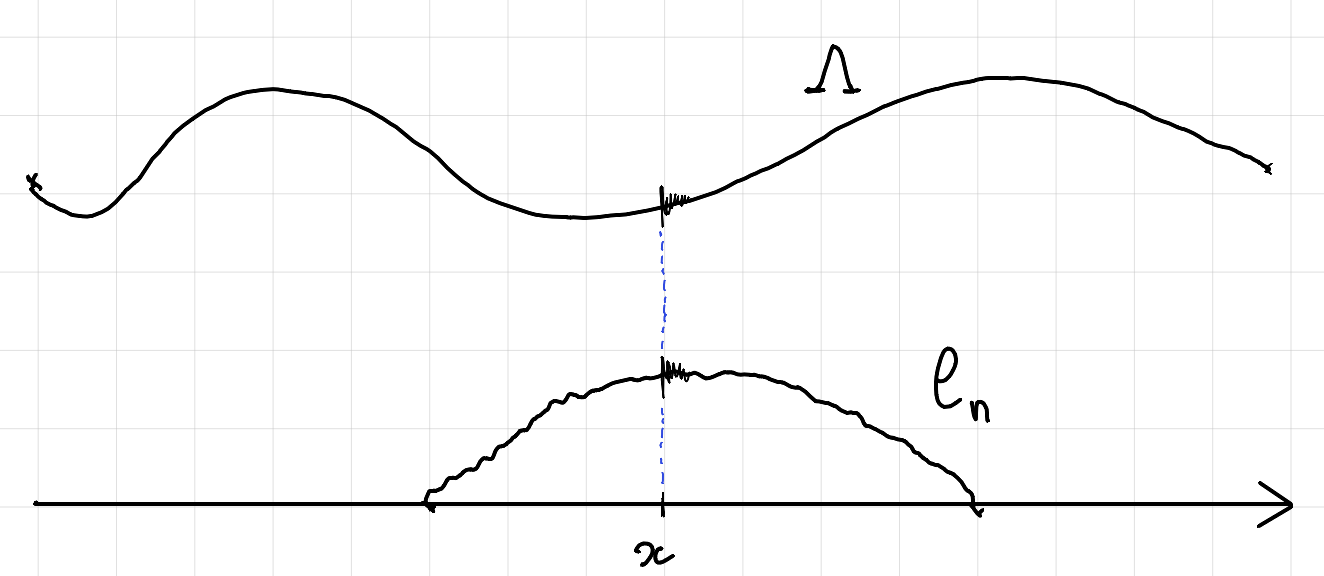}
    \caption{Sketch of the occupation time densities $\Lambda$ and $\ell_n$ of the whole loop-soup and of one loop $\beta_n$ in the loop-soup, on some edge of the cable-graph~: Any $1$-quick point of $\Lambda$ in the support of $\ell_n$ is necessarily a quick point of $\ell_n$. }
    \label{figure1}
\end{figure}

Let $\ell_n$ denote the density of the occupation time measure of $\beta_n$. Then clearly,
$\Lambda = \Lambda_n + \ell_n$. Lemma \ref {Ladding} (or rather its immediate generalization to cable-graphs) implies that almost surely, for every $n$, any point $x$ that is a $1$-quick point of $\Lambda$ with $\Lambda_{n} (x) >0 $ and $\ell_n (x) > 0$ is necessarily a quick point of $\ell_n$. We can note that the set of $1$-quick points of $\Lambda$ is almost surely dense, so it will be almost surely dense on any excursions set away from $0$ by $\Lambda_{n}$ and $\ell_n$. Moreover, it is known that almost surely, $\ell_n$ is positive on the entire range of the loop $\beta_n$ (with the exception of its finitely many boundary points).

Hence, we can conclude that the conditional law of $\beta_n$ given $\Lambda$ has the property that almost surely, 
for every 1-quick point $x$ of $\Lambda$ in the interior of the range of $\beta_n$, either $x$ is a quick point of $\ell_n$ as sketched on Figure~\ref{figure1}, or $\ell_n(x) = \Lambda(x)$ (it is not hard to check that in that case too $x$ has to be a quick point of $\ell_n$, but this is not necessary here). But for any deterministic point $x$ on the cable-graph, almost all loops according to the (unconditioned) Brownian loop-measure on the cable-graph have the property that their occupation measure has no quick point at $x$ and that their local time at $x$ is almost always different from $\Lambda(x)$, so that the law of the loop $\beta_n$ conditionally on $\Lambda$ is indeed singular with respect to that of an unconditioned loop (i.e., it is singular with respect to the unconditioned Brownian loop measure).

\begin {remark}
To prove that one can detect that a Brownian loop is not an unconditioned loop by looking at only a small portion of it, there are a number of options. One way is to use the rewiring property of \cite {werner2015spatialmarkovpropertysoups} at some given point [i.e., one chooses two times at which the loop-soup visits this point, and resamples whether these two times are part of the same loop or not -- mind one can choose this point as a function of $\Lambda$ also, and that the rewiring procedure for $c \not=1$ exists as well -- the rewiring probabilities just depend on the number of loops the rewiring creates] that shows that when one applies a well-chosen rewiring transformation, any portion of any loop will contain what will become a loop in the modified sample, and one can then apply the previous result to this modified loop-soup -- so that this portion will need to create quick points at special points determined by $\Lambda$.
\end {remark}

\begin {remark}
Brownian loop-soups are interesting to study also in continuum $d$-dimensional spaces when $d>1$ -- the definition in \cite {MR2045953} was in fact motivated by the connection with loop-erasures of Brownian motion (for instance, the set of erased loops in Wilson's algorithm correspond to a $c=2$ loop-soup) which can be formulated also in the continuum, or with the connection with SLE processes in two dimensions. However, in such cases, the question that we investigate in the present paper becomes almost trivial. Indeed, the trace of the loop-soup is easily seen to have zero Lebesgue measure in the plane, so that when one conditions the loop-soup on its trace, each Brownian motion is conditioned to stay in this trace and is clearly singular with respect to an unconditioned Brownian loop. So, the ``singularity'' question is essentially a one-dimensional one.
\end {remark}

\section {Singularity of the Brownian burglar}
\label {proof2}

We now conclude this note by briefly explaining how to implement the same ideas to prove the result about the Brownian burglar mentioned in the introduction:

Let us consider a Brownian motion $\beta$ on $\R$. In fact, it will be more convenient to consider $\beta$ to be a reflected Brownian motion -- taking therefore its values in $[0, \infty)$.
Let $\ell_t (x)$ denote its local time at position $x$ and time $t$, and let $\tau(h)$ be the first time at which $\ell_t (0)$ reaches $h$, which is a stopping time for each $h>0$.
The classical Ray-Knight Theorem states that for each $h>0$, the process $(\ell_{\tau(h)} (x))_{x \ge 0}$ is a squared Bessel process of dimension $0$ started from $h$. Recall that squared Bessel processes are absorbed at $0$.

The strong Markov property of Brownian motion at times $\tau(h)$ immediately implies that for any $0<h_1 < \ldots < h_m$, the processes $\ell_{\tau (h_1)}, \ell_{\tau (h_2)} - \ell_{\tau (h_1)},
\ldots , \ell_{\tau (h_m)} - \ell_{\tau (h_{m-1})}$ are independent squared Bessel processes of dimension $0$ started from $h_1$, $h_2- h_1$, \ldots, $h_{m}- h_{m-1}$ respectively.
[This is of course consistent with the classical feature of squared Bessel processes: The sum of two independent squared Bessel processes of dimensions $d$ and $d'$ is a squared Bessel of dimension $d$ + $d'$ -- we are here in the case where $d=d'=d+d' = 0$].

We can actually consider the case where $h_j = j2^{-n}$ and let $n \to \infty$ to see that if we now define for each $h$,
$$ L_h (\cdot) := \inf_{g>h} \ell_{\tau(g)} (\cdot) - \ell_{\tau(h)} (\cdot), $$
then the process $(L_h)_{h \ge 0}$ is a Poisson point process of excursions of squared Bessel processes of dimension 0 away from $0$.

This corresponds exactly to the decomposition of $\beta$ into its excursions away from the origin. Indeed, if $e_h$ denotes the Poisson point process of excursions away from the origin by $\beta$ (i.e., $e_h$ is the excursion when the local time at the origin is $h$), then $L_h$ will be the local time profile of this excursion.

Then:

\begin {proposition}
\label {Pburglar}
For all $\eps >0$ and $l >0$, the conditional law of $(\beta_t, t \le \eps)$ given the occupation time measure $(\ell_{\tau(l)} (x))_{x \ge 0}$ at time $\tau(l)$ is singular with respect to that of an unconditioned Brownian motion.
\end {proposition}

\begin {proof}
For $0 < h <l$, let us define $X_1 (x):= \ell_{\tau (h)} (x)$ and $X_2 (x) := \ell_{\tau(l)} (x) - \ell_{\tau(h)}(x)$.  These are two independent squared Bessel processes of dimension $0$ that are respectively started from $h$ and $l-h$. It follows that almost surely, every $1$-quick point $x$ of $\ell_{\tau(l)}$ at which $X_1$ is positive is necessarily a quick point for $X_1$ (if $X_2$ is positive it is Lemma~\ref {Ladding}, otherwise $X_1(y) = \ell_{\tau(l)}(y)$  for $y\geq x$). If $(x_k)$ is a sequence of $1$-fast points that goes to $0$, then it follows that for the conditional law, almost surely, for all large enough $k$, $x_k$ is a fast point for $X_1$. But for the unconditional law of a squared Bessel process of dimension $0$, almost surely, none of the $x_k$'s will be a quick point, which shows that for all $x_0$, the conditional law of $(\ell_{\tau (h)} (x), x \le x_0)$ is indeed singular with respect to the unconditioned one. But we also know that $\tau(h) \to 0$ almost surely as $h \to 0$ -- so that the proposition follows easily.
\end {proof}

\begin {remark}
It is possible to adapt the ideas of the proof (or alternatively to use this proposition) to show that a similar result holds true when one replaces $\tau(l)$ by any other stopping time.
\end {remark}

\subsection*{Acknowledgements}
I thank Wendelin Werner for help and advice during the preparation of this note. I also thank Elie A\"\i d\'ekon for his comments on the first version of this paper. This research has been funded by a grant from the Royal Society.

\bibliographystyle{plain}

\end{document}